\documentclass[a4paper, 11pt]{article}

\usepackage{amssymb}
\usepackage{amsmath}
\usepackage{amsthm}
\usepackage{enumerate}
\usepackage[mathscr]{eucal}
\usepackage[all]{xy}
\usepackage{graphics}
\usepackage[dvips]{graphicx}
\usepackage{float}
\usepackage[usenames,dvipsnames]{color}

\begin{document}

\def\myend{{}\hfill{\small$\bigcirc$}}
\newfloat{diagram}{H}{figure}
\floatname{diagram}{Figure}

\let\goth\mathfrak

\newcommand\theoremname{Theorem}
\newcommand\lemmaname{Lemma}
\newcommand\corollaryname{Corollary}
\newcommand\propositionname{Proposition}
\newcommand\factname{Fact}
\newcommand\remarkname{Remark}
\newcommand\examplename{Example}

\newtheorem{thm}{\theoremname}[section]
\newtheorem{lem}[thm]{\lemmaname}
\newtheorem{cor}[thm]{\corollaryname}
\newtheorem{prop}[thm]{\propositionname}
\newtheorem{fact}[thm]{\factname}
\newtheorem{exmx}[thm]{\examplename}
\newenvironment{exm}{\begin{exmx}\normalfont}{\end{exmx}}
\newtheorem{rem}[thm]{\remarkname}


\newcounter{sentence}
\def\thesentence{\roman{sentence}}
\def\labelsentence{\upshape(\thesentence)}
\def\bigtimes{\mbox{\Large$\times$}}
\newenvironment{sentences}{%
        \list{\labelsentence}
          {\usecounter{sentence}\def\makelabel##1{\hss\llap{##1}}
            \topsep3pt\leftmargin0pt\itemindent40pt\labelsep8pt}%
  }{%
    \endlist}

\newenvironment{ctext}{%
  \par
  \smallskip
  \centering
}{%
 \par
 \smallskip
 \csname @endpetrue\endcsname
}
\newenvironment{cmath}{%
  \par
  \smallskip
  \centering
  $
}{%
  $
  \par
  \smallskip
  \csname @endpetrue\endcsname
}


\newcommand*{\sub}{\raise.5ex\hbox{\ensuremath{\wp}}}
\newcommand{\nat}{{N}}   
\newcommand{\msub}{\mbox{\large$\goth y$}}    
\def\id{\mathrm{id}}
\def\suport{{\mathrm{supp}}}
\def\Aut{\mathrm{Aut}}
\newcommand*{\struct}[1]{{\ensuremath{\langle #1 \rangle}}}

\def\bagnom#1#2{\genfrac{[}{]}{0pt}{}{#1}{#2}}

\def\kros(#1,#2){c_{\{ #1,#2 \}}}
\def\skros(#1,#2){{\{ #1,#2 \}}}
\def\LineOn(#1,#2){\overline{{#1},{#2}\rule{0em}{1,5ex}}}
\def\inc{\mathrel{\strut\rule{3pt}{0pt}\rule{1pt}{9pt}\rule{3pt}{0pt}\strut}}
\def\lines{{\cal L}}
\def\tran{{\mathscr T}}
\def\collin{\sim}
\def\wspolin{{\bf L}}
\def\ncollin{\not\collin}
\def\ciach#1{\upharpoonright_{\;{#1}}}
\def\mappedby#1{%
{\rule{0pt}{2.5ex}}\mkern8mu{\longmapsto\mkern-25mu{\raise1.3ex\hbox{$#1$}}}
\mkern20mu{\rule{0pt}{4pt}}
}


\title{On some generalization of the M\"obius configuration}

\author{Krzysztof Petelczyc}

\pagestyle{myheadings}
\markboth{K. Petelczyc}{Generalization of the M\"obius configuration}

\maketitle

\bigskip

\par\noindent{\small
Author's address:\\
Krzysztof Petelczyc\\
Institute of Mathematics, University of Bia{\l}ystok\\
ul. Akademicka 2\\
15-246 Bia{\l}ystok, Poland\\
e-mail: {\ttfamily kryzpet@math.uwb.edu.pl}}

\bigskip

\begin{abstract}
The M\"obius $(8_4)$ configuration is generalized in a purely combinatorial approach. We consider 
$(2n_n)$ configurations ${\goth M}_{(n,\varphi)}$ depending on a permutation $\varphi$ in the symmetric
group $S_n$. Classes of non-isomorphic configurations of this type are determined.
The parametric characterization of ${\goth M}_{(n,\varphi)}$ is given. 
The uniqueness of the decomposition of ${\goth M}_{(n,\varphi)}$ into two mutually inscribed $n$-simplices is
discussed.
The automorphisms of ${\goth M}_{(n,\varphi)}$ are characterized for $n\geq 3$.
\end{abstract}

\section*{Introduction}
The M\"obius $(8_4)$ configuration is a certain configuration in a projective $3$-dimensional space consisting
of two mutually inscribed and circumscribed tetrahedra (cf. \cite{moebius}). Each vertex of one tetrahedron lies on a face plane of the other tetrahedron and vice versa. 
The axiom connected with the M\"obius configuration plays the same role as
the Pappus axiom in a projective $3$-dimensional space (cf. \cite{witczyn}): it is equivalent to the commutativity of the ground division ring.
\par
In this paper we deal with two $n$-simplices (simplices with $n$ vertices, $n\geq 3$) instead of two
tetrahedra ($4$-simplices). The way how $n$-simplex is inscribed into another we define by a permutation $\varphi$ in the group $S_n$. The generalization of the M\"obius configuration we obtain, is a $(2n_n)$-configuration and it will be referred to as a {\em M\"obius pair of $n$-simplices}, or shortly a {\em M\"obius $n$-pair}. Only a combinatorial scheme of a M\"obius $n$-pair is
investigated and we do not discuss problems regarding embeddability into projective spaces. Although these problems have been partially solved in \cite{havlik} (the case with $\varphi=\id$), they are interesting and still open in general.
\par
As we know from \cite{hilbert}, in a projective space, up to an isomorphism there are five $(8_4)$ point-plane configurations with the property that at most two planes share two points, and dually at most two points are shared by two planes. These are precisely those configurations with two mutually circumscribed tetrahedra, and thus all of them are sometimes called the M\"obius configurations.
It is also known, that these $(8_4)$ configurations correspond to conjugacy classes of the permutation group $S_4$. We shall prove, that two
M\"obius $n$-pairs are isomorphic if and only if the permutations, that determine them, are conjugate. Another important impact of the permutation on the geometry of the M\"obius $n$-pair is that
the cycle structure of $\varphi$ is associated with circuits in the incidence graph of the M\"obius $n$-pair. 
\par
As we shall see, the decomposition of the points of the generalized M\"obius configuration into two complementary and mutually inscribed simplices is, generally, a unique one. Exceptions appear ``near" the classical case $n=4$. Three of five $(8_4)$ M\"obius configurations contain at least two distinct pairs of complementary $4$-simplices.
\par
The next problem, which is considered in the paper, involves M\"obius subpairs of a M\"obius $n$-pair. We simply delete some number of points and blocks of one $n$-simplex and the same number of points and blocks of the second $n$-simplex with a hope to obtain a M\"obius pair again. The conditions, under which we get a subpair in the M\"obius $n$-pair, are determined. 
\par
In the last part we use most of the established properties to characterize the automorphism group of the M\"obius $n$-pair for $n\geq 3$.


\section{Definitions, parameters and basic properties}

By {\em configuration} we mean any point-block structure, where blocks are subsets of the set of points.
Let ${\sub}_{k}(X)$ stand for a family of all
$k$-subsets of the set $X$.
The {\em rank of a point} is the number of blocks containing this point, and dually the {\em size of a block}
is the number of points contained in this block.
We say that two configurations ${\goth M}_1=\struct{S_1,\lines_1}, {\goth M}_2=\struct{S_2,\lines_2}$ 
are {\em isomorphic} (and we write ${\goth M}_1\cong {\goth M}_2$) iff there exists a bijective map $f\colon S_1\longrightarrow S_2$ such that
conditions $k\in\lines_1$ and $f(k)\in \lines_2$ are equivalent. In case ${\goth M}_1={\goth M}_2=\goth M$ the map $f$
will be called an {\em automorphism of $\goth M$}.
\par
Let us consider two sets 
$A=\{a_1,\ldots,a_n\}$ and $B=\{b_1,\ldots, b_n\}$ such that $A\cap B=\emptyset$ and $n\geq 3$.
Let $\varphi\in S_n$ be a permutation of the set 
$ I =\{1,\ldots,n\}$. 
Now we introduce the following sets:
$$\begin{array}{lll}
\lines_A & := & \{A'\cup\{b_i\}\colon A'\in {\sub}_{n-1}(A) \text{ and } a_i\notin A'\}, \\
\lines_B & := & \{B'\cup\{a_{\varphi(i)}\}\colon B'\in \sub_{n-1}(B) \text{ and } b_i\notin B'\}.
\end{array}$$
The configuration 
 $${\goth M}_{(n,\varphi)}:=\struct{A\cup B,\lines_A\cup \lines_B},$$
will be called a {\em M\"obius $n$-pair}.
The M\"obius configurations can be identified with the M\"obius $4$-pairs (see Figures \ref{levi:moeb_id}, \ref{levi:moeb_1cyc}, \ref{levi:moeb_2cyc13}, \ref{levi:moeb_3cyc},
\ref{levi:moeb_2cyc22}). In particular,
${\goth M}_{(4,\id)}$ is the classical $(8_4)$ M\"obius configuration. 
\par
Let $\cal M$ be a M\"obius $n$-pair.
We write: $A_i,B_i$ for blocks of $\cal M$ not containing $a_i, b_i$, respectively; 
{\em $a$-points}, {\em $b$-points}, {\em $A$-blocks}, {\em $B$-blocks} for points in $A$, $B$, and blocks in 
$\lines_A$, $\lines_B$, respectively.
The configuration $\cal M$ reflects main abstract properties of the classical M\"obius configuration.
\begin{itemize}\itemsep-2pt
\item
The $a$-points yield a simplex in $\cal M$: for any $(n-1)$-subset $A\setminus\{a_i\}$ of the $a$-points
there is a (unique) block of $\cal M$, which contains this subset ($A_i$, a face of the simplex in question); the remaining points ($b$-points) yield another simplex.
\item
The simplex with $a$-points and the simplex with $b$-points are mutually inscribed:
on each face, $A_i$, of the first simplex there is a unique vertex ($b_i$) of the second one; on each face, $B_i$, of the second simplex there is a unique vertex ($a_{\varphi(i)}$) of the first simplex. 
\end{itemize}
Thus, we can decompose $\cal M$ into two complementary substructures $S_A({\cal M})=\struct{A,\lines_A}$ and 
$S_B({\cal M})=\struct{B,\lines_B}$, which we call {\em simplices of $\cal M$} (although, formally, a block of each of them is not a subset of its points; there is one extra point on each of its faces). 
\par
In the forthcoming part we will use the notion of the incidence graph (the Levi graph) ${\cal G}_{\cal M}$
associated with $\cal M$.
Recall that a Levi graph is a bipartite graph with partition induced by points vs. blocks.
Two of its vertices $x,y$ are said to be {\em adjacent} (which is written $x\sim y$) if $x$ is a point, $y$ is a block (or vice versa) and $x\in y$ (or $y\in x$). Otherwise $x$ is not adjacent to $y$, which we write $x\nsim y$.
The {\em rank of a vertex} is the number of vertices adjacent to it.
A vertex of ${\cal G}_{\cal M}$ will be called point-vertex, block-vertex, $a$,$b$-vertex, $A$,$B$-vertex, or simply $a_i,b_i,A_i,B_i$ as it corresponds to the point or to the block of ${\cal M}$.
The Levi graph associated with $S_A({\cal M}), S_B({\cal M})$ will be denoted by ${\cal G}_{S_A({\cal M})}, {\cal G}_{S_B({\cal M})}$, respectively.

\begin{rem}\label{rem:Levi}
Let $\cal M$ be a M\"obius $n$-pair.
The Levi graph ${\cal G}_{\cal M}$ has the following properties:
\begin{sentences}\itemsep-2pt
\item\label{one-simpl}
for $X=A,B$,
every point-vertex from ${\cal G}_{S_X({\cal M})}$
is adjacent to all but one block-vertices from ${\cal G}_{S_X({\cal M})}$, and vice versa, 
\item\label{two-simpl}
for $X,Y=A,B$ and $X\neq Y$,
every point-vertex from ${\cal G}_{S_X({\cal M})}$ is adjacent to precisely one
block-vertex from ${\cal G}_{S_Y({\cal M})}$, and vice versa.
\end{sentences}
\end{rem}
Immediately from the definition of ${\goth M}_{(n,\varphi)}$, the number of its points coincides with the number of its blocks and equals $2n$, and the rank of every point coincides with the size of every block and equals $n$. Thus the
structures we investigate are $(2n_n)$-configurations.
A standard parametric question related to configurations is: what is the number of points that are contained in two distinct blocks,  
and dually: what is the number of blocks containing two distinct points.
\begin{prop}\label{lem:blockcap}
Let $k,l$ be two different blocks of the structure ${\goth M}_{(n,\varphi)}$.
Then $|k\cap l|\in \{0,1,2,n-2\}$. If both $k$, $l$ are $A$-blocks, or both
$k$, $l$ are $B$-blocks then $|k\cap l|=n-2$. Otherwise, 
 $k=A_i$ and $l=B_j$ for some
$i,j\in I$, and the following equivalences hold
\begin{sentences}\itemsep-2pt
\item\label{case:zero}
$|A_i\cap B_j|=0$ iff $\varphi(j)=i=j$,
\item\label{case:one}
$|A_i\cap B_j|=1$ iff $\varphi(j)=i\neq j$ or $\varphi(j)\neq i= j$,
\item\label{case:two}
$|A_i\cap B_j|=2$ iff $\varphi(j)\neq i\neq j$.
\end{sentences}
\end{prop}
\begin{proof}
It is straightforward from the definition that if $k,l$ are both $A$-blocks or $B$-blocks then $k\cap l$ has $n-2$ elements. Let $k=A_i\in\lines_A$ and $l=B_j\in\lines_B$ for some $i,j\in I$. Let $i\neq j$.
If $\varphi(j)\neq i$ then $A_i\cap B_j=\{b_i, a_{\varphi(j)}\}$.
Otherwise, for $\varphi(j)=i$, we get $A_i\cap B_j=\{b_i\}$.
Let $i=j$. If $\varphi(i)\neq i$ we obtain $A_i\cap B_i=\{a_{\varphi(i)}\}$. 
In case $\varphi(i)=i$ it holds $A_i\cap B_i=\emptyset$.
\end{proof}

Each conjugacy class of $S_n$ corresponds to exactly one decomposition of a permutation $\varphi\in S_n$
into cycles, up to a permutation of the elements of $I$. 
Now we describe how the cycle structure of $\varphi$ is reflected
in block paths of ${\goth M}_{(n,\varphi)}$.
\begin{fact}\label{fact:cycles}
A permutation $\varphi$ contains a cycle of length $k\leq n$ iff there is a path
of length $2k$ consisting of blocks of ${\goth M}_{(n,\varphi)}$ such that, every two consecutive blocks
intersect in precisely one point of ${\goth M}_{(n,\varphi)}$.

\end{fact}
\begin{proof} 
Assume that $\varphi$ contains the cycle $({i_1} {i_2}\ldots {i_k})$. 
Then $a_{i_{j+1}}\in A_{i_j}\cap B_{i_j}$ and $b_{i_{j+1}}\in B_{i_j}, A_{i_{j+1}}$ for each $j\leq k$.
Thus, the path in question is the following:
$A_{i_1}$, $B_{i_1}$, $A_{i_2}$, $B_{i_2}$,
\ldots, $A_{i_k}$, $B_{i_k}$.
\par
Now assume that there exists a closed path $l_1,l'_1,\ldots,l_{k},l'_{k}$
of blocks of ${\goth M}_{(n,\varphi)}$ such that, every two consecutive blocks
intersect in a point. By \ref{lem:blockcap}\eqref{case:one} every two consecutive blocks of the path are 
$A_{i}\in\lines_A$, $B_{j}\in\lines_B$ with $\varphi(j)=i\neq j$ or $\varphi(j)\neq i= j$.
Suppose $\varphi(j)\neq i= j$ holds for the first two blocks of our path, namely $l_1=A_{i}$, $l'_1=B_{i}$ and 
$\varphi(i)\neq i$ for some $i\in I$. To obtain $|l'_1\cap l_2|=1$ we must have $l_2=A_{j}$ with
$\varphi(i)=j$. Thus the next two blocks are
$l_2=A_{\varphi(i)}$, $l'_2=B_{\varphi(i)}$ and
$\varphi(\varphi(i))\neq \varphi(i)$. In general we obtain 
$l_{j}=A_{\varphi^{j-1}(i)}$, $l'_{j}=B_{\varphi^{j-1}(i)}$ and
$\varphi^{j-1}(i)\neq \varphi^{j-2}(i)$ for every $j=2,\ldots,k$.
To close the path we need $\varphi^{k}(i)=i$. Let us put $i=i_0$.
Then the cycle $(i_0,i_1,\ldots,i_{k-1})$, where $i_j=\varphi^j(i)$ for
$j=0,\ldots,k-1$, is one of the cycles in the cycle decomposition of $\varphi$.
\end{proof}
As the configuration ${\goth M}_{(n,\varphi)}$ is symmetric, it makes sense to consider
the dual configuration ${\goth M}^\circ_{(n,\varphi)}$. 
\begin{fact}\label{dualdes}
The configuration
${\goth M}^\circ_{(n,\varphi)}$ is isomorphic to ${\goth M}_{(n,\varphi)}$.
\end{fact}
\begin{proof}
It is easy to note that ${\goth M}^\circ_{(n,\varphi)}\cong {\goth M}_{(n,\varphi^{-1})}$. 
Consider $\alpha\in S_n$ such that $\alpha(1)=1$ and $\alpha(m)=n-m+2$ for $m\in I\setminus\{1\}$.
Let $x\in\{a,b,A,B\}$, $i\in I$.
Then $F: x_i\mapsto x_{\alpha(i)}$ is an isomorphism mapping ${\goth M}_{(n,\varphi^{-1})}$ onto
${\goth M}_{(n,\varphi)}$.
\end{proof}
The problem of two isomorphic M\"obius $n$-pairs will be consider in general in the last section of the paper.
Another parametric characterization is now a simple consequence of \ref{lem:blockcap} and \ref{dualdes}.
\begin{prop}
Let $x,y$ be two different points of ${\goth M}_{(n,\varphi)}$.
There exist $0$, $1$, $2$, or $n-2$ blocks of ${\goth M}_{(n,\varphi)}$ containing $x$ and $y$.
\end{prop}

\begin{diagram}
\centering
\[
\xymatrixrowsep{0.5in}
\xymatrixcolsep{0.2in}
\xymatrix{%
{a_1}\ar@{-}[dr]\ar@{-}[drr]\ar@{-}[drrr]\ar@{-}[drrrrr]
&
{a_2}\ar@{-}[dl]\ar@{-}[dr]\ar@{-}[drr]\ar@{-}[drrrrr]
&
{a_3}\ar@{-}[dll]\ar@{-}[dl]\ar@{-}[dr]\ar@{-}[drrrrr]
&
{a_4}\ar@{-}[dl]\ar@{-}[dll]\ar@{-}[dlll]\ar@{-}[drrrrr]
&
&
{b_1}\ar@{-}[dr]\ar@{-}[drr]\ar@{-}[drrr]\ar@{-}[dlllll]
&
{b_2}\ar@{-}[dl]\ar@{-}[dr]\ar@{-}[drr]\ar@{-}[dlllll]
&
{b_3}\ar@{-}[dll]\ar@{-}[dl]\ar@{-}[dr]\ar@{-}[dlllll]
&
{b_4}\ar@{-}[dl]\ar@{-}[dll]\ar@{-}[dlll]\ar@{-}[dlllll]
\\
{A_1}
&
{A_2}
&
{A_3}
&
{A_4}
&
&
{B_1}
&
{B_2}
&
{B_3}
&
{B_4}
}\]
\caption{The Levi graph of ${\goth M}_{(4,\id)}$.}
\label{levi:moeb_id}
\end{diagram}

\begin{diagram}
\centering
\[
\xymatrixrowsep{0.5in}
\xymatrixcolsep{0.2in}
\xymatrix{%
{a_1}\ar@{-}[dr]\ar@{-}[drr]\ar@{-}[drrr]\ar@{-}[drrrrrrrr]
&
{a_2}\ar@{-}[dl]\ar@{-}[dr]\ar@{-}[drr]\ar@{-}[drrrr]
&
{a_3}\ar@{-}[dll]\ar@{-}[dl]\ar@{-}[dr]\ar@{-}[drrrr]
&
{a_4}\ar@{-}[dl]\ar@{-}[dll]\ar@{-}[dlll]\ar@{-}[drrrr]
&
&
{b_1}\ar@{-}[dr]\ar@{-}[drr]\ar@{-}[drrr]\ar@{-}[dlllll]
&
{b_2}\ar@{-}[dl]\ar@{-}[dr]\ar@{-}[drr]\ar@{-}[dlllll]
&
{b_3}\ar@{-}[dll]\ar@{-}[dl]\ar@{-}[dr]\ar@{-}[dlllll]
&
{b_4}\ar@{-}[dl]\ar@{-}[dll]\ar@{-}[dlll]\ar@{-}[dlllll]
\\
{A_1}
&
{A_2}
&
{A_3}
&
{A_4}
&
&
{B_1}
&
{B_2}
&
{B_3}
&
{B_4}
}\]
\caption{The Levi graph of ${\goth M}_{(4,\varphi)}$ with $\varphi=(1234)$.}
\label{levi:moeb_1cyc}
\end{diagram}

\begin{diagram}
\centering
\[
\xymatrixrowsep{0.5in}
\xymatrixcolsep{0.2in}
\xymatrix{%
{a_1}\ar@{-}[dr]\ar@{-}[drr]\ar@{-}[drrr]\ar@{-}[drrrrrrr]
&
{a_2}\ar@{-}[dl]\ar@{-}[dr]\ar@{-}[drr]\ar@{-}[drrrr]
&
{a_3}\ar@{-}[dll]\ar@{-}[dl]\ar@{-}[dr]\ar@{-}[drrrr]
&
{a_4}\ar@{-}[dl]\ar@{-}[dll]\ar@{-}[dlll]\ar@{-}[drrrrr]
&
&
{b_1}\ar@{-}[dr]\ar@{-}[drr]\ar@{-}[drrr]\ar@{-}[dlllll]
&
{b_2}\ar@{-}[dl]\ar@{-}[dr]\ar@{-}[drr]\ar@{-}[dlllll]
&
{b_3}\ar@{-}[dll]\ar@{-}[dl]\ar@{-}[dr]\ar@{-}[dlllll]
&
{b_4}\ar@{-}[dl]\ar@{-}[dll]\ar@{-}[dlll]\ar@{-}[dlllll]
\\
{A_1}
&
{A_2}
&
{A_3}
&
{A_4}
&
&
{B_1}
&
{B_2}
&
{B_3}
&
{B_4}
}\]
\caption{The Levi graph of ${\goth M}_{(4,\varphi)}$ with $\varphi=(123)(4)$.}
\label{levi:moeb_2cyc13}
\end{diagram}

\begin{diagram}
\centering
\[
\xymatrixrowsep{0.5in}
\xymatrixcolsep{0.2in}
\xymatrix{%
{a_1}\ar@{-}[dr]\ar@{-}[drr]\ar@{-}[drrr]\ar@{-}[drrrrr]
&
{a_2}\ar@{-}[dl]\ar@{-}[dr]\ar@{-}[drr]\ar@{-}[drrrrr]
&
{a_3}\ar@{-}[dll]\ar@{-}[dl]\ar@{-}[dr]\ar@{-}[drrrrrr]
&
{a_4}\ar@{-}[dl]\ar@{-}[dll]\ar@{-}[dlll]\ar@{-}[drrrr]
&
&
{b_1}\ar@{-}[dr]\ar@{-}[drr]\ar@{-}[drrr]\ar@{-}[dlllll]
&
{b_2}\ar@{-}[dl]\ar@{-}[dr]\ar@{-}[drr]\ar@{-}[dlllll]
&
{b_3}\ar@{-}[dll]\ar@{-}[dl]\ar@{-}[dr]\ar@{-}[dlllll]
&
{b_4}\ar@{-}[dl]\ar@{-}[dll]\ar@{-}[dlll]\ar@{-}[dlllll]
\\
{A_1}
&
{A_2}
&
{A_3}
&
{A_4}
&
&
{B_1}
&
{B_2}
&
{B_3}
&
{B_4}
}\]
\caption{The Levi graph of ${\goth M}_{(4,\varphi)}$ with $\varphi=(1)(2)(34)$.}
\label{levi:moeb_3cyc}
\end{diagram}

\begin{diagram}
\centering
\[
\xymatrixrowsep{0.5in}
\xymatrixcolsep{0.2in}
\xymatrix{%
{a_1}\ar@{-}[dr]\ar@{-}[drr]\ar@{-}[drrr]\ar@{-}[drrrrrr]
&
{a_2}\ar@{-}[dl]\ar@{-}[dr]\ar@{-}[drr]\ar@{-}[drrrr]
&
{a_3}\ar@{-}[dll]\ar@{-}[dl]\ar@{-}[dr]\ar@{-}[drrrrrr]
&
{a_4}\ar@{-}[dl]\ar@{-}[dll]\ar@{-}[dlll]\ar@{-}[drrrr]
&
&
{b_1}\ar@{-}[dr]\ar@{-}[drr]\ar@{-}[drrr]\ar@{-}[dlllll]
&
{b_2}\ar@{-}[dl]\ar@{-}[dr]\ar@{-}[drr]\ar@{-}[dlllll]
&
{b_3}\ar@{-}[dll]\ar@{-}[dl]\ar@{-}[dr]\ar@{-}[dlllll]
&
{b_4}\ar@{-}[dl]\ar@{-}[dll]\ar@{-}[dlll]\ar@{-}[dlllll]
\\
{A_1}
&
{A_2}
&
{A_3}
&
{A_4}
&
&
{B_1}
&
{B_2}
&
{B_3}
&
{B_4}
}\]
\caption{The Levi graph of ${\goth M}_{(4,\varphi)}$ with $\varphi=(12)(34)$.}
\label{levi:moeb_2cyc22}
\end{diagram}

\section{Hidden M\"obius pairs}

The goal of this section is to characterize ${\cal M}={\goth M}_{(n,\varphi)}$
that can be transformed into M\"obius pair with simplices distinct from
$S_A({\cal M})$, $S_B({\cal M})$ by a decomposition of the points or by deletion of some points and blocks. 
Informally, we say that these M\"obius pairs are {\em hidden} in ${\cal M}$.

\subsection{M\"obius $n$-pairs with the special decompositions}

Let us start with the following combinatorial observation
\begin{rem}\label{ilepar}
The M\"obius configuration ${\cal M}={\goth M}_{(4,\id)}$ can be presented in $3$ distinct ways as two mutually circumscribed simplices, such that each of them is distinct from $S_A({\cal M}), S_B({\cal M})$.
\end{rem}
One could say that there are four M\"obius $4$-pairs hidden in ${\goth M}_{(4,\id)}$. 
Let $n\geq 4$, ${\cal M}={\goth M}_{(n,\varphi)}$, and assume that it is possible to decompose the points
of $\cal M$ into two complementary and mutually inscribed simplices $S_1({\cal M})$, $S_2({\cal M})$, such that
$S_t({\cal M})\neq S_X({\cal M})$ for each $t=1,2$, $X=A,B$. Since 
this kind of decomposition is rather not evident for M\"obius $n$-pairs, it will be called a {\em special decomposition}. 
\begin{lem}\label{lem:notstable}
Let $S_1({\cal M})$, $S_2({\cal M})$ be two simplices, that arise from a special decomposition of $\cal M$.
\begin{sentences}\itemsep-2pt 
\item\label{slupek}
 For each $i\in I$, and each $t=1,2$, it is impossible to have both $B_i$,$b_i$ in $S_t({\cal M})$, or both $A_i$,$a_i$ in $S_t({\cal M})$.
\item\label{twoAtwoB}
For each $t=1,2$, blocks of $S_t({\cal M})$ are two $B$-blocks and two $A$-blocks.
\end{sentences}
\end{lem}
\begin{proof}
The proof involves only $S_1({\cal M})$, since the reasoning for $S_2({\cal M})$ will be the same.

(i)
Assume that $S_1({\cal M})$ contains both of $B_i,b_i$. Then also
some $a_j$ is a point of $S_1({\cal M})$ for $j\in I$. Consider the graph ${\cal G}_{\cal M}$.
The vertices $B_i,b_i$ are not adjacent, so from \ref{rem:Levi}\eqref{one-simpl} $a_j\sim B_i$ and $j=\varphi(i)$. The unique block-vertex not adjacent to $a_j$ in ${\cal G}_{S_1({\cal M})}$ is $A_j$ or $B_s$ for some $s\neq \varphi^{-1}(j)$.
\newline
Let $A_j$ be this vertex, so from \ref{rem:Levi}\eqref{two-simpl} $A_j\sim b_i$, and thus $j=i$. Consider in ${\cal G}_{S_1({\cal M})}$ another vertex $a_t$ or $b_t$ with $t\neq i$.
Since $\varphi(i)=i\neq t$, a contradiction arises: $b_t\nsim A_i$, and $a_t\nsim B_i$ (see the scheme presented in Figure \ref{levi:slupek1}).   

\begin{minipage}[m]{0.4\textwidth}
\begin{diagram}
\begin{center}
\xymatrix{%
{b_i}\ar@{--}[d]\ar@{-}[dr]
&
{a_i=a_{\varphi(i)}}\ar@{--}[d]\ar@{-}[dl]
&
{a_t}\ar@{--}[dll]
&
{b_t}\ar@{--}[dll]
\\
{B_i}
&
{A_i}
}
\end{center}
\caption{The fragment of ${\cal G}_{f(\cal M)}$ containing $B_i,b_i$ and $A_i,a_i$.}
\label{levi:slupek1}
\end{diagram}
\end{minipage}
\hfil
\begin{minipage}[m]{0.4\textwidth}
\begin{diagram}
\begin{center}
\xymatrix{%
{b_i}\ar@{--}[d]\ar@{-}[dr]\ar@{-}[drr]
&
{a_{\varphi(i)}}\ar@{--}[d]\ar@{-}[dl]\ar@{-}[dr]
&
{b_r}\ar@{--}[d]\ar@{-}[dll]\ar@{-}[dl]
\\
{B_i}
&
{B_s}
&
{A_i}
}
\end{center}
\caption{The fragment of ${\cal G}_{f(\cal M)}$ containing $B_i,b_i$ and $B_s,a_{\varphi(i)}$.}
\label{levi:slupek2}
\end{diagram}
\end{minipage}

\bigskip

\noindent
Assume that $s\neq \varphi^{-1}(j)$ and $B_s$ is the unique vertex not adjacent to $a_j$ in ${\cal G}_{S_1({\cal M})}$. 
We get $s\neq i$, as far as $b_i\sim B_s$. Let us take another vertex: $A_t$ or $B_t$. For $t\neq i$ there is no $B$-vertex adjacent to $a_j$, and $A_t\sim a_j,b_i$ if $t=i$.
A vertex, which is not adjacent to $A_i$, is $a_i$ or $b_r$ with $r\neq i,s$. The vertex $a_i$ is not adjacent to $B_i$ since $\varphi(i)=j\neq i$, and thus $a_i$ cannot be the vertex in question. Consequently, this vertex is $b_r\sim B_i,B_s$. Following the assumption $n\geq 4$, there exists another block in $S_1({\cal M})$, that is different from $B_i,B_s,A_i$.
We have two $b$-points in $S_1({\cal M})$ so far, thus this block is a $B$-block. The $B$-vertex of ${\cal G}_{\cal M}$, that is associated with this block, must be adjacent to
$a_{\varphi(i)}$. So this block is $B_i$, which is already one of the blocks in $S_1({\cal M})$ (comp. with the scheme presented in Figure \ref{levi:slupek2}), a contradiction.

(ii)
Let $B_i$ be the unique $B$-block of $S_1({\cal M})$ for some $i\in I$.
Then the remaining blocks of $S_1({\cal M})$ are $A$-blocks. In view of lemma \ref{lem:notstable}\eqref{slupek}, there are $n-1$ $b$-vertices in ${\cal G}_{S_1({\cal M})}$: every $A$-vertex is associated with the $b$-vertex, which is not adjacent to it. For $n\geq 4$ a contradiction with \ref{rem:Levi}\eqref{one-simpl} arises: every $b$-vertex is adjacent to precisely one of $A$-vertices, and thus it is not adjacent to at least two $A$-vertices in ${\cal G}_{S_1({\cal M})}$. 
\newline
Let $S_1({\cal M})$ contain at least three $B$-blocks. Without loss of generality, assume $B_1,B_2,B_3$ are blocks of $S_1({\cal M})$.
From \ref{lem:notstable}\eqref{slupek}, $b_1,b_2,b_3$ are not in $S_1({\cal M})$. Thus, from \ref{rem:Levi}\eqref{one-simpl}, 
$S_1({\cal M})$ contains $a_{i_1},a_{i_2},a_{i_3}$ such that $i_j\neq \varphi(j)$ for $j=1,2,3$. Every block-vertex $B_j$ must be adjacent to at least two of the point-vertices $a_{i_{j'}}$ with $j'\neq j$. On the other hand, it is adjacent to at most one of them, what follows from \ref{rem:Levi}\eqref{two-simpl} applied to ${\cal G}_{\cal M}$. This contradiction actually completes the proof as other cases run dually.
\end{proof}
By \ref{lem:notstable} we prove a generalization of \ref{ilepar}.
\begin{prop}\label{nonoriginal}
Let ${\cal M}={\goth M}_{(n,\varphi)}$.
The following conditions are equivalent
\begin{sentences}\itemsep-2pt
\item
 there is a special decomposition of $\cal M$ , 
\item
$n=4$ and there is $X\subset I$, such that $|X|=2$ and $\varphi(X)=X$.
\end{sentences}
\end{prop}
\begin{proof}
(i)$\Rightarrow$(ii):
From \ref{lem:notstable}\eqref{twoAtwoB} we get $n=4$, and two $B$-vertices and two $A$-vertices in
${\cal G}_{S_1({\cal M})}$. Let (e.g.) $B_1,B_2$ be the $B$-vertices of ${\cal G}_{S_1({\cal M})}$. 
In view of \ref{rem:Levi}\eqref{one-simpl}, there are vertices $x,y$ in ${\cal G}_{S_1({\cal M})}$ such that $x\sim B_1, y\sim B_2$ and $x\nsim B_2, y\nsim B_1$. By \ref{lem:notstable}\eqref{slupek}, $x\neq b_2$, $y\neq b_2$, and thus 
$x=a_i$, $y=a_j$ where $\varphi(1)=i,\varphi(2)=j$. Then two $A$-vertices in  ${\cal G}_{S_1({\cal M})}$  are 
$A_s,A_t$ with $s,t\neq i,j$. The remaining two point-vertices must be of the form $b_{s'},b_{t'}$ with $s',t'\neq 1,2$, since they must be adjacent to both of $B_1,B_2$. On the other hand, $b_{s'},b_{t'}$ need to be adjacent to precisely one of $A_s,A_t$, so $\{s',t'\}=\{s,t\}$. Thus $s,t\neq 1,2$, $\{1,2\}=\{i,j\}=\{\varphi(1),\varphi(2)\}$, and $X=\{1,2\}$ is the required set.

(ii)$\Rightarrow$(i):
Assume, without loss of generality, $X=\{1,2\}$ and consider ${\cal M}={\goth M}_{(4,\varphi)}$ with $\varphi(X)=X$. Take blocks $B_1,B_2,A_3,A_4$ and points $a_{\varphi(1)},a_{\varphi(2)},b_3,b_4$ of $\cal M$, and consider ${\cal G}_{\cal M}$.
We have $B_1\nsim a_{\varphi(2)}$, $B_2\nsim a_{\varphi(1)}$, and $B_1,B_2\sim b_3,b_4$. Similarly
$A_3\nsim b_4$, $A_4\nsim b_3$, and $A_3,A_4\sim a_{\varphi(1)},a_{\varphi(2)}$,
since $\varphi(1),\varphi(2)\in\{1,2\}$. Thus the Levi graph we consider is a Levi graph of a $4$-simplex. It is easy to verify that $A_1,A_2,B_3,B_4$ and
$b_1,b_2,a_3,a_4$ form another $4$-simplex. The two obtained simplices are mutually circumscribed. Indeed, $B_1,b_2$; $B_2,b_1$; $A_3,a_4$; $A_4,a_3$, and $A_1,a_{\varphi(1)}$ (or $A_1,a_{\varphi(2)}$); $A_2,a_{\varphi(2)}$
(or $A_2,a_{\varphi(1)}$); $B_3,b_4$; $B_4,b_3$ are all pairs of adjacent vertices representing blocks (points) of the first simplex and points (blocks) of the second simplex in each pair. In other words, we have found a special decomposition of $\cal M$.
\end{proof}
Due to \ref{nonoriginal} there is a correspondence between the special decompositions of ${\goth M}_{(n,\varphi)}$
and $2$-subsets of $I$ preserved by $\varphi$. The correspondence is established up to complements, since the
special decompositions arise only for $n=4$, and thus if $\varphi$ preserves a $2$-subset of $\{1,2,3,4\}$ then
it preserves its complement as well. So, directly from \ref{nonoriginal} we get
\begin{cor}\label{cor:all-s-labile}
All (up to an isomorphism) M\"obius $n$-pairs with a special decomposition are the following:
\begin{enumerate}\itemsep-2pt
\item
 ${\goth M}_{(4,\id)}$ with $3$ distinct special decompositions associated with \\
$X=\{1,2\},\{1,3\},\{1,4\}$,
\item
 ${\goth M}_{(4,(13)(24))}$ with the special decomposition associated with $X=\{1,3\}$,
\item
 ${\goth M}_{(4,(12)(3)(4))}$ with the special decomposition associated with $X=\{1,2\}$.
\end{enumerate}
\end{cor}


\subsection{Subpairs of M\"obius $n$-pairs}
Let ${\cal M}={\goth M}_{(n,\varphi)}$, $n\geq 4$, 
$k\geq 3$, $k<n$, and ${\cal M}'$ be a M\"obius $k$-pair obtained from ${\cal M}$ by deleting $2(n-k)$ points and $2(n-k)$ blocks. We call ${\cal M}'$ a {\em $k$-subpair} of ${\cal M}$.
Blocks of ${\cal M}'$ are {\em subblocks of $\cal M$}, that is every
block of ${\cal M}'$ arises as a block of $\cal M$ with $n-k$ points removed. Subblocks of
$A$-blocks, $B$-blocks are called {\em $A$-subblocks}, {\em $B$-subblocks}, respectively. 
Let $S_1({\cal M}')$, $S_2({\cal M}')$ be two simplices of ${\cal M}'$. For any $t=1,2$, $X=A,B$ we write $S_t({\cal M}')\prec S_X({\cal M})$ if all points and blocks of $S_t({\cal M}')$ are points and subblocks of $S_X({\cal M})$.
Otherwise we write $S_i({\cal M}')\nprec S_X({\cal M})$.
\par
In order to determine all M\"obius $n$-pairs with $k$-subpairs we need to prove some auxiliary facts.
\begin{lem}\label{lem:tylkob}
One of the following conditions holds
\begin{sentences}\itemsep-2pt
\item\label{case1}
$S_1({\cal M}')\prec S_A({\cal M})$ and $S_2({\cal M}')\prec S_B({\cal M})$,
\item\label{case11}
$S_2({\cal M}')\prec S_A({\cal M})$ and $S_1({\cal M}')\prec S_B({\cal M})$,
\item\label{case2}
$S_1({\cal M}')\nprec S_A({\cal M}), S_B({\cal M})$ and $S_2({\cal M}')\nprec S_A({\cal M}), S_B({\cal M})$.
\end{sentences}
Moreover, if $\cal M'$ satisfies \eqref{case2} then there is a special decomposition of $\cal M'$. 
\end{lem}
\begin{proof}
Let $S_1({\cal M}')\prec S_A({\cal M})$ and $S_2({\cal M}')\nprec S_B({\cal M})$.
So there is an $a$-point or $A$-subblock in $S_2({\cal M}')$. We consider only the case with an $a$-point, as the case with an $A$-subblock is symmetric.
From \ref{rem:Levi}\eqref{two-simpl} applied to ${\cal G}_{\cal M}$, and \ref{rem:Levi}\eqref{one-simpl}
applied to ${\cal G}_{{\cal M}'}$, there are at most two $B$-subblocks in $S_2({\cal M}')$. 
Since $k\geq 3$, there is at least one $A$-subblock in $S_2({\cal M}')$. 
Note that the unique $A$-subblock, which does not contain an $a$-point of $S_1({\cal M}')$, 
is the block of $S_1({\cal M}')$. 
Thus all points of $S_1({\cal M})$ are in an $A$-subblock of $S_2({\cal M}')$. This yields a contradiction with \ref{rem:Levi}\eqref{two-simpl}.
 The proof for each of the remaining cases (i.e. $S_2({\cal M}')\prec S_A({\cal M})$ and $S_1({\cal M}')\nprec S_B({\cal M})$, $S_1({\cal M}')\prec S_B({\cal M})$ and $S_2({\cal M}')\nprec S_A({\cal M})$, or $S_2({\cal M}')\prec S_B({\cal M})$ and $S_1({\cal M}')\nprec S_A({\cal M})$) is analogous. 
\par 
Let $\cal M'$ satisfy \eqref{case2}. The steps of the proof of \ref{lem:notstable} can be repeated for simplices
of $\cal M'$. As a result we get $k=4$, and two $A$-subblocks and two $B$-subblocks in each of simplices
of $\cal M'$. Let $Y\subset I$ be the set of subscripts of $A$-subblocks and $B$-subblocks in one of these simplices. 
From the reasoning analogous to the first part of the proof of \ref{nonoriginal} we get that
$Y$ is the set of all subscripts used for labelling points and blocks of $\cal M'$, and there is a
two-element set $X\subset Y$ such that $\varphi\ciach{Y}(X)=X$.
Therefore, in a view of \ref{nonoriginal}, there is a special decomposition of $\cal M'$.
\end{proof}
\begin{lem}\label{fixedset}
If the number of deleted $B$-blocks and the number of deleted $A$-blocks coincide (and equals $n-k$),
then there is $X\subset I$, such that $|X|=n-k$ and $\varphi(X)=X$.
\end{lem}
\begin{proof}
Assume that $B_{i_1},\ldots,B_{i_{n-k}}$ and $A_{j_1},\ldots,A_{j_{n-k}}$ are removed blocks.
Consider a vertex $a_{\varphi(i_s)}$ with
$s=1,\ldots,n-k$ of ${\cal G}_{\cal M'}$, and assume $a_{\varphi(i_s)}$ is in ${\cal G}_{S_1(\cal M')}$ (the case with
$a_{\varphi(i_s)}$ in ${\cal G}_{S_2(\cal M')}$ will be analogous). 
Note that $a_{\varphi(i_s)}\sim B_{i_s}$, and from \ref{rem:Levi}\eqref{two-simpl} $B_{i_s}$ is the unique $B$-vertex
adjacent to $a_{\varphi(i_s)}$.
According to \ref{lem:tylkob} two cases arise: \eqref{case1} or \eqref{case2} holds for
${\cal M}'$.
Let $\cal M'$ satisfy \eqref{case1} of \ref{lem:tylkob}. 
Then there is a $B$-vertex in ${\cal G}_{S_2(\cal M')}$ adjacent to $a_{\varphi(i_s)}$, a contradiction.
If $\cal M'$ satisfies \eqref{case2} of \ref{lem:tylkob} then there is a special decomposition of $\cal M'$.
So, by \ref{nonoriginal}, there is
a $B$-vertex in ${\cal G}_{S_1(\cal M')}$ adjacent to $a_{\varphi(i_s)}$, a contradiction again.
Therefore all $a_{\varphi(i_1)}\ldots,a_{\varphi(i_{n-k})}$ are removed. 
Likewise we consider the pairs $a_{j_s},A_{j_s}$, $b_{j_s},A_{j_s}$, and
$b_{i_s}, B_{i_s}$. Each of these reasonings leads us to contradiction.
Consequently
points $a_{j_1}\ldots,a_{j_{n-k}}$, $b_{j_1}\ldots,b_{j_{n-k}}$, and $b_{i_1}\ldots,b_{i_{n-k}}$
are deleted as well.
Hence
\begin{cmath}
 \{j_1,\ldots,j_{n-k}\}=\{\varphi(i_1),\ldots,\varphi(i_{n-k})\} \text{ and } \{i_1,\ldots,i_{n-k}\}=\{j_1,\ldots,j_{n-k}\}.
\end{cmath}
Finally we get $\{\varphi(i_1),\ldots,\varphi(i_{n-k})\}=\{i_1,\ldots,i_{n-k}\}$, and $X=\{i_1,\ldots,i_{n-k}\}$ is the set from our claim.
\end{proof}
Let us present a condition, which is sufficient and necessary to find a $k$-subpair in ${\goth M}_{(n,\varphi)}$.
\begin{prop}\label{pairnested}
Let ${\cal M}={\goth M}_{(n,\varphi)}$.
The following conditions are equivalent
\begin{sentences}\itemsep-2pt
\item
there is ${\cal M'}$, which is a $k$-subpair of $\cal M$,
\item\label{pairnested:2}
there is $X\subset I$, such that $|X|=n-k$ and $\varphi(X)=X$.
\end{sentences}
Furthermore, if \eqref{pairnested:2} holds then ${\cal M'}\cong {\goth M}_{(k,\varphi\ciach{(I\setminus X)})}$. 
\end{prop}
\begin{proof}
(i)$\Rightarrow$(ii):
By \ref{lem:tylkob} $\cal M'$ satisfies one of \ref{lem:tylkob}\eqref{case1} -- \ref{lem:tylkob}\eqref{case2}.
In cases \ref{lem:tylkob}\eqref{case1}, \ref{lem:tylkob}\eqref{case11} the numbers of $A$-blocks and $B$-blocks deleted from $\cal M$ coincide and are equal $n-k$. The claim follows directly from \ref{fixedset}.
If \ref{lem:tylkob}\eqref{case2} holds, then there is a special decomposition of $\cal M'$, and we get our claim by \ref{nonoriginal}.

(ii)$\Rightarrow$(i):
Without any loss of generality, let $X=\{1,\ldots,n-k\}$.
Recall that the rank of every vertex in ${\cal G}_{\cal M}$ is $n$.
Observe the Levi graph obtained from ${\cal G}_{\cal M}$ by removing the  
vertices $a_i, A_i$ and $b_i,B_i$ for every $i\in X$, and all edges passing through these vertices. We denote 
this Levi graph by ${\cal H}$. Note that $a_{\varphi(i)}$ is not a vertex of $\cal H$, since $\varphi(i)\in X$.
Let $j\notin X$ and take $A_j$. Clearly $A_j$ is a vertex of $\cal H$.
There are $n-k$ edges joining $A_j$ with all $a_i$ in ${\cal G}_{\cal M}$. 
Thus, the rank of $A_j$ in $\cal H$ is $n-(n-k)=k$. Similarly we set ranks of the remaining vertices $a_j,b_j,B_j$ of
$\cal H$. All these ranks are $k$. From this and the construction of ${\cal H}$ we get that
${\cal H}$ is the Levi graph of two mutually circumscribed $k$-simplices, where the way they are inscribed one into another is induced by the action of $\varphi$ on the set $I \setminus X$. Therefore ${\cal H}={\cal G}_{\cal M'}$
for some $\cal M'$, which is a $k$-subpair of $\cal M$.
\end{proof}
%
%


\section{Isomorphisms and automorphisms}

\subsection{Isomorphic M\"obius $n$-pairs}

Recall that the M\"obius $(8_4)$ configurations (i.e. M\"obius $4$-pairs) correspond to conjugacy classes of the permutation group $S_4$. In this section we generalize this property to all M\"obius $n$-pairs.
\par
Let us start with a key lemma that gives an account on isomorphisms of configurations ${\goth M}_{(n,\varphi)}$ with the unique decomposition into two $n$-simplices.
\begin{lem}\label{lem:isoform}
Let $f$ be an isomorphism mapping ${\goth M}_{(n,\varphi)}$ onto ${\goth M}_{(n,\psi)}$.
Assume that either
$n=4$ and both $\varphi,\psi\neq \id$ contain no cycle of length $2$, or
 $n\geq 5$. 
There is $\alpha\in S_n$ such that $f(B_i)=B_{\alpha(i)}$ for each $i\in I$, or $f(B_i)=A_{\alpha(i)}$ for each $i\in I$.
\begin{sentences}\itemsep-2pt
\item\label{lem:isoform:i}
If $f(B_i)=B_{\alpha(i)}$ then $f(b_i)=b_{\alpha(i)}$, $f(A_i)=A_{\alpha(i)}$, $f(a_i)=a_{\alpha(i)}$ for each $i\in I$.
\item\label{lem:isoform:ii}
If
$f(B_i)=A_{\alpha(i)}$ then $f(b_i)=a_{\alpha(i)}$, $f(A_i)=B_{\psi^{-1}(\alpha(i))}$, $f(a_i)=b_{\psi^{-1}(\alpha(i))}$
for each $i\in I$.
\end{sentences}
Furthermore, $\alpha\varphi=\psi\alpha$ holds in both cases: \eqref{lem:isoform:i} and \eqref{lem:isoform:ii}. 
\end{lem}
\begin{proof}
Let ${\cal M}_1:={\goth M}_{(n,\varphi)}$ and ${\cal M}_2:={\goth M}_{(n,\psi)}$. 
Let $i,j\in I$ and $B_i$ be an arbitrary
$B$-block of ${\cal M}_1$. Clearly, either $f(B_i)=B_j$ for some $B$-block $B_j$ of ${\cal M}_2$, or $f(B_i)=A_j$ for some  $A$-block $A_j$ of ${\cal M}_2$. 
\par
Assume that $f(B_i)=B_j$.
In view of \ref{cor:all-s-labile}, both ${\cal M}_1$, ${\cal M}_2$ are M\"obius $n$-pairs without the special decompositions. Thus all $B$-blocks of ${\cal M}_1$ are mapped onto $B$-blocks of ${\cal M}_2$. 
We introduce a map $\alpha\in S_n$ associated 
with $f$ by the formula
\begin{ctext}
 $\alpha\colon i\mapsto j$ iff $f(B_i)=B_j$, 
\end{ctext}
for all $i,j\in I$.
Then $f(B_i)=B_{\alpha(i)}$.
Let us analyze graphs ${\cal G}_{{\cal M}_1}$ and
${\cal G}_{{\cal M}_2}$:
$f(b_i)=b_{\alpha(i)}$ as $b_i$, $b_{\alpha(i)}$ are the unique $b$-vertices not adjacent to $B_i$, $B_{\alpha(i)}$ 
respectively in graphs ${\cal G}_{{\cal M}_1}$, ${\cal G}_{{\cal M}_2}$;
$f(A_i)=A_{\alpha(i)}$ as $A_i$, $A_{\alpha(i)}$ are the unique $A$-vertices adjacent to $b_i$, $b_{\alpha(i)}$ 
respectively in ${\cal G}_{{\cal M}_1}$, ${\cal G}_{{\cal M}_2}$;
$f(a_i)=a_{\alpha(i)}$ as $a_i$, $a_{\alpha(i)}$ are the unique $a$-vertices not adjacent to $A_i$, $A_{\alpha(i)}$ 
respectively in ${\cal G}_{{\cal M}_1}$, ${\cal G}_{{\cal M}_2}$.
On the other hand, $f(a_{\varphi(i)})=a_{\psi(\alpha(i))}$ as $a_{\varphi(i)}$, $a_{\psi(\alpha(i))}$ are the unique $a$-vertices adjacent to $B_i$, $B_{\alpha(i)}$ respectively in ${\cal G}_{{\cal M}_1}$, ${\cal G}_{{\cal M}_2}$.
So $a_{\alpha(\varphi(i))}=a_{\psi(\alpha(i))}$ and thus $\alpha\varphi=\psi\alpha$. 
\par 
In case $f(B_i)=A_j$ the map
$\alpha\in S_n$ is determined by the condition 
\begin{ctext}
$\alpha\colon i\mapsto j$ iff $f(B_i)=A_j$,
\end{ctext}
for all $i,j\in I$. Then
we proceed in a similar way as in the former case, namely:
$f(b_i)=a_{\alpha(i)}$ as $b_i$, $a_{\alpha(i)}$ are the unique $b$-vertex and $a$-vertex not adjacent to $B_i$, $A_{\alpha(i)}$ respectively in ${\cal G}_{{\cal M}_1}$, ${\cal G}_{{\cal M}_2}$;
$f(A_i)=B_{\psi^{-1}(\alpha(i))}$ as $A_i$, $B_{\psi^{-1}(\alpha(i))}$ are the unique $A$-vertex and $B$-vertex adjacent to $b_i$, $a_{\alpha(i)}$ respectively in ${\cal G}_{{\cal M}_1}$, ${\cal G}_{{\cal M}_2}$;
$f(a_i)=b_{\psi^{-1}(\alpha(i))}$ as $a_i$, $b_{\psi^{-1}(\alpha(i))}$ are the unique $a$-vertex and $b$-vertex not adjacent to $A_i$, $B_{\psi^{-1}(\alpha(i))}$ respectively in ${\cal G}_{{\cal M}_1}$, ${\cal G}_{{\cal M}_2}$.
But also $f(a_{\varphi(i)})=b_{\alpha(i)}$ as $a_{\varphi(i)}$, $b_{\alpha(i)}$ are the unique $a$-vertex and $b$-vertex
adjacent to $B_i$, $A_{\alpha(i)}$ respectively in ${\cal G}_{{\cal M}_1}$, ${\cal G}_{{\cal M}_2}$. Hence
$b_{\psi^{-1}(\alpha(\varphi(i)))}=b_{\alpha(i)}$, and consequently $\alpha\varphi=\psi\alpha$.
\end{proof}
We are ready to characterize two isomorphic M\"obius $n$-pairs.  
\begin{thm}\label{conclass:strong}
Let $n\geq 4$ and $\varphi,\psi,\alpha\in S_n$. The following conditions are equivalent:
\begin{sentences}\itemsep-2pt
\item
$\varphi^\alpha=\psi$,
\item
${\goth M}_{(n,\varphi)}\cong {\goth M}_{(n,\psi)}$.
\end{sentences}
\end{thm}
\begin{proof}
Let ${\cal M}_1={\goth M}_{(n,\varphi)}$ and ${\cal M}_2={\goth M}_{(n,\psi)}$.

(i)$\Rightarrow$(ii):
Let $i\in I$, $a_i,b_i$ be points and $A_i,B_i$ be blocks of ${\cal M}_1$.
Consider a map $f$ associated to the permutation $\alpha$ given by the formula
\begin{ctext}
$f(x_i) = x_{\alpha(i)}$ for $x\in\{a,b\}$,
\end{ctext}
which maps the points of ${\cal M}_1$ onto the points of ${\cal M}_2$.
Then $f(A_i)=A_{\alpha(i)}$ and $f(B_i)=B_{\alpha(i)}$, as the conditions $a_i\notin A_i$, $b_i\notin B_i$ uniquely determine blocks $A_i$, $B_i$, respectively. Clearly, conditions $b_i\in A_i$ and $b_{\alpha(i)}\in A_{\alpha(i)}$ are 
equivalent.
Note that $a_{\alpha(\varphi(i))}\in B_{\alpha(i)}$ is equivalent to $a_{\psi(\alpha(i))}\in B_{\alpha(i)}$ as well, since
$\alpha\varphi=\psi\alpha$. Thus $f$ is the required isomorphism.

(ii)$\Rightarrow$(i): 
We restrict ourselves to $n\geq 5$ since for $n=4$ this fact is well known, as it was mentioned at the beginning of
this section.
Let $f$ be an isomorphism mapping ${\cal M}_1$ onto ${\cal M}_2$.
By \ref{lem:isoform}, there is $\alpha\in S_n$ associated with $f$ such that $\alpha\varphi=\psi\alpha$.
%
\end{proof}
According to \ref{conclass:strong}, the number of non-isomorphic configurations ${\goth M}_{(n,\varphi)}$ is equal to the number of partitions $p(n)$ of a positive integer $n$. There is the generating function, recursive formula, asymptotic formula, and direct formula for $p(n)$ (cf. \cite{andrews}). The increase of $n$ implies quick growth of $p(n)$: $p(5) = 7$, $p(6) = 11$,\ldots,
$p(100) = 190569292$,\ldots, $p(1000)=24061467864032622473692149727991$.


\subsection{The automorphism group structure of a M\"obius $n$-pair}

For $n=3$ the structure ${\goth M}_{(n,\varphi)}$ consists of two simultaneously inscribed and described triangles. From
\cite{petel} the automorphism group of ${\goth M}_{(3,\varphi)}$ is isomorphic to
$S_3\ltimes C_2$.
From the original paper of M\"obius \cite{moebius} the automorphism group of
${\goth M}_{(4,\id)}$ has order $192$. The M\"obius configuration is also a
particular case of the Cox configuration. Recall the definition of the Cox configuration. Let $X$ be a set with $n$ elements. The incidence structure
\begin{ctext}
$(\mathbf{Cx})_X=(\mathbf{Cx})_n=\struct{\bigcup\{{\sub}_{2k+1}(X)\colon 0\leq k\leq n\}, \{{\sub}_{2k}(X)\colon 0\leq k\leq n\}, \subset \cup \supset}$
\end{ctext}
is the ${(2^{n-1}}_n)$ configuration, which is called the Cox configuration.
Since the automorphism group of $(\mathbf{Cx})_n$
is established in \cite{cox} and ${\goth M}_{(4,\id)}=(\mathbf{Cx})_4$ (see Figure \ref{fig:coxmoeb}), we get the following.
\begin{figure}[!h]
  \begin{center}
    \includegraphics[scale=0.6]{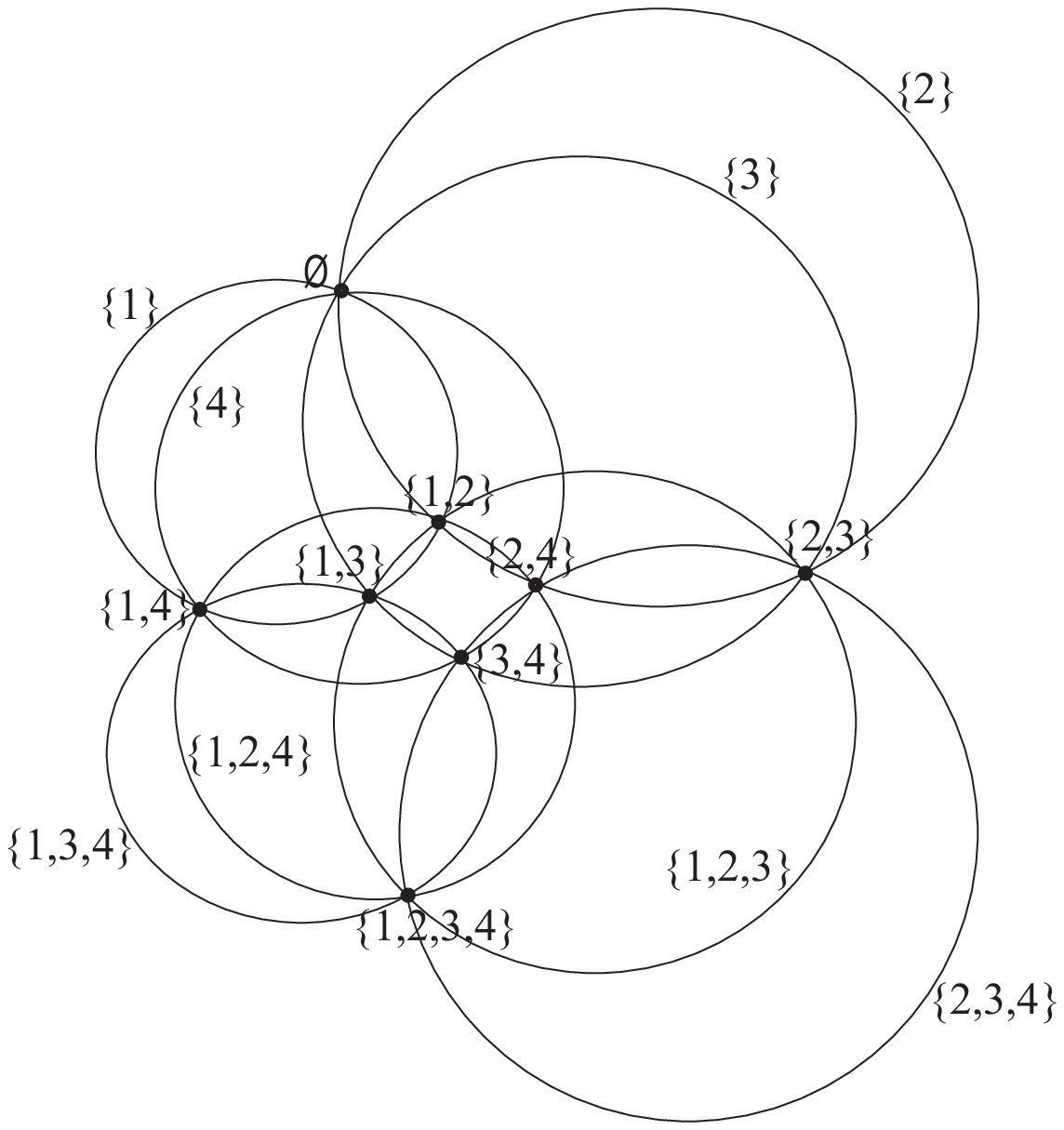}
  \end{center}
  \caption{The M\"obius configuration as $(\mathbf{Cx})_4$.}
  \label{fig:coxmoeb}
\end{figure}
\begin{fact}\label{autofmoebius}
The automorphism group of ${\goth M}_{(4,\id)}$ is isomorphic to  $S_4 \ltimes C_2^{3}$.
\end{fact}
%

It follows from \ref{conclass:strong} that the centralizer
of $\varphi$ in $S_n$ consists of automorphisms of ${\goth M}_{(n,\varphi)}$ for any $n$.
Nevertheless, we will give a detailed characterization of automorphism group of ${\goth M}_{(4,\varphi)}$ with $\varphi\neq\id$, and of ${\goth M}_{(n,\varphi)}$ with $n\geq 5$. 
\par
Let ${\cal M}={\goth M}_{(n,\varphi)}$ and
 $1\leq\nu_1<\ldots<\nu_r$ be the lengths of the cycles which are contained in the cycle decomposition of $\varphi\in S_n$. Assume that there are $m_t$ cycles of length $\nu_t$, so $n=\sum_{t=1}^{r}m_t\nu_t$.
In other words 
$$\varphi=\varphi^{\nu_1}_{1}\varphi^{\nu_1}_2\ldots\varphi^{\nu_1}_{m_1}\varphi^{\nu_2}_{1}\varphi^{\nu_2}_2\ldots\varphi^{\nu_2}_{m_2}\ldots
\varphi^{\nu_r}_{1}\varphi^{\nu_r}_2\ldots\varphi^{\nu_r}_{m_r},$$
where $\varphi^{\nu_t}_{k}$ is a cycle of length $\nu_t$ for $k\leq m_t$, $t\leq r$.
In view of \ref{conclass:strong} we can assume, that each cycle consists of consecutive natural numbers. 
If we set $\mu_k^t:=\sum_{i=1}^{t-1}m_i\nu_i+(k-1)\nu_t+1$ then $$\varphi^{\nu_t}_{k}\colon\mu_k^t \mapsto \mu_k^t +1 \mapsto \mu_k^t + 2\mapsto ...\mapsto \mu_k^t + (\nu_t-1) \mapsto \mu_k^t,$$
 and the effective domain of $\varphi^{\nu_t}_{k}$ is the set $X^{\nu_t}_k:=\{\mu_k^t,\mu_k^t+1 \ldots,\mu_k^t + (\nu_t-1)\}\subseteq I$.
Taking all the domains of all cycles we obtain the family of pairwise disjoint sets
$X^{\nu_1}_1,\ldots,X^{\nu_1}_{m_1},X^{\nu_2}_1,\ldots,X^{\nu_2}_{m_2},\ldots,X^{\nu_r}_1,\ldots,X^{\nu_r}_{m_r}$
that yields a covering of $I$.
Thus for any cycle $\varphi^{\nu_t}_{k}$ we have $\varphi^{\nu_t}_{k}(X^{\nu_t}_{k})=X^{\nu_t}_{k}$ and
$\varphi^{\nu_t}_{k} \ciach{I\setminus X^{\nu_t}_k}=\id$.
\par
The points and the blocks of $\cal M$ can be identified with the sequences  $(t,k,i,\varepsilon)$ such that $t\leq r$, $k\leq m_t$, $i=0,\ldots,\nu_t-1$, and $\varepsilon\in\{1,2,-1,-2\}$ according to the formula:
\begin{equation}\label{label:points}
(t,k,i,\varepsilon)=
\left\{
\begin{array}{llll}
a_{i+\mu_k^t} & \text {for } \varepsilon=1, \\
b_{i+\mu_k^t} & \text{for } \varepsilon=-1, \\
A_{i+\mu_k^t} & \text {for } \varepsilon=2, \\
B_{i+\mu_k^t} & \text {for } \varepsilon=-2.
\end{array}
\right.
\end{equation}%
Let $v_t=(v_1^t,\ldots,v^t_{m_t})\in C_{\nu_t}^{m_t}$, $\alpha_t\in S_{m_t}$, and
$v=(v_1,\ldots,v_r)\in\bigtimes_{t=1}^{r}C_{\nu_t}^{m_t}$, 
$\alpha=(\alpha_1,\ldots,\alpha_r)\in\bigtimes_{t=1}^{r}S_{m_t}$.
With the pair
$(v,\alpha)$ we associate the map $f_{(v,\alpha)}$ as follows:
\begin{equation}\label{map:pres}
f_{(v,\alpha)}((t,k,i,\varepsilon))=
(t,\alpha_t(k), i+v_k^t\mod\nu_t,\varepsilon).
\end{equation}
In like manner we define the map $g_{(v,\alpha)}$ by:
\begin{equation}\label{map:inter}
g_{(v,\alpha)}((t,k,i,\varepsilon))=
\left\{
\begin{array}{ll}
(t,\alpha_t(k),i+v_k^t-1\mod\nu_t,-\varepsilon) & \text {for } \varepsilon=1,2, \\
(t,\alpha_t(k),i+v_k^t\mod\nu_t,-\varepsilon) & \text {for } \varepsilon=-1,-2.
\end{array}
\right.
\end{equation}
\begin{lem}\label{presautomorf}
The map $f_{(v,\alpha)}$ is an automorphism of $\cal M$, which preserves each of simplices
$S_A$, $S_B$.
\end{lem}
\begin{proof}
It follows directly from \eqref{map:pres}, that
$f_{(v,\alpha)}$ maps $S_A$ onto $S_A$, and $f_{(v,\alpha)}$ maps $S_B$ onto $S_B$.
Let $i\in X^{\nu_t}_k$ and $j\in I$.
\par
Assume that $b_j\in B_i$. By \eqref{label:points},
$B_{i}=(t,k,i_0,-2)$ for some $i_0\in\{0,\ldots,\nu_t-1\}$, and $b_j=(t',k',j_0,-1)$ for some
$t'\leq r$, $k'\leq m_{t'}$, $j_0\in\{0,\ldots,\nu_{t'}-1\}$.
Then $f(B_{i})=(t,\alpha_t(k),i_0+v_{\alpha_t(k)}^t\mod\nu_t,-2)$ and $f(b_{j})=(t',\alpha_{t'}(k'),j_0+v_{\alpha_{t'}(k')}^{t'}\mod\nu_{t'},-2)$.
Set $i'=(i_0+v_{\alpha_t(k)}^t\mod\nu_t)+\mu^t_{\alpha_t(k)}$ and
$j'=(j_0+v_{\alpha_{t'}(k')}^{t'}\mod\nu_{t'})+\mu^{t'}_{\alpha_{t'}(k')}$,
so $f(B_{i})=B_{i'}$ and $f(b_{j})=B_{j'}$.
Recall that $b_j\in B_i$ iff $j\neq i$. If $j'\neq i'$ then: firstly $t'=t$, next $\alpha_t(k')=\alpha_t(k)$ and thus $k'=k$, and finally $j_0=i_0$. It means that $j=i$, which yields a contradiction.
Hence $f(b_j)\in f(B_{i})$.
\par
Let $a_j \in B_i$. Then $j=\varphi(i)$.
We have $a_{\varphi(i)}=(t,k,i_0+1\mod \nu_t,1)$, so $f(a_{\varphi(i)})=(t,k,i_0+1+v_k^t\mod \nu_t,1)=
a_{\varphi(i')}$. Therefore
$f(a_{\varphi(i)})\in f(B_{i})$. 
\par
The incidence (membership) relation
is preserved by $f_{(v,\alpha)}$ in case $a_j\in A_i$ and in case $b_j\in A_i$ as well, that can be easily proved by similar reasoning.
\end{proof}
Let $\mathbf{v}_t=(\underbrace{v,\ldots,v}_{t})$ for all $t\leq r$, and  $\mathbf{v}=(\mathbf{v}_1,\ldots,\mathbf{v}_r)$. Let us put $g_0:=g_{(\mathbf{0},\id)}$. 
\begin{lem}\label{autBontoA}
The map $g_0$ is an automorphism of $\cal M$, which interchanges simplices
$S_A$, $S_B$.
\end{lem}
\begin{proof}
Immediately from \eqref{map:pres}, $g_0$ maps $S_A$ onto $S_B$, and $S_B$ onto $S_A$.
We restrict our proof to the incidence relation involving $B$-blocks, as the case with $A$-blocks runs similarly.
Let $i\in X^{\nu_t}_{k}$. From \eqref{label:points} $B_i$ is represented by the sequence $(t,k,i_0,-2)$ for some $i_0\in\{0,\ldots,\nu_t-1\}$. The points that belongs to $B_{i}$ are $b_j$ with $j\in I\setminus\{i\}$ and $a_{\varphi(i)}$.
Clearly, $g_0(b_j)=a_j\in A_{i}=g_0(B_{i})$. We have $a_{\varphi(i)}=(t,k,i_0+1\mod\nu_t,1)$ and thus $g_0(a_{\varphi(i)})=(t,k,i_0,-1)=b_{i}$. Then finally $g_0(a_{\varphi(i)})\in g_0(B_{i})$. 
\end{proof}
Since $g_{(v,\alpha)}=g_0f_{(v,\alpha)}$, from \ref{presautomorf} and \ref{autBontoA} we infer that
\begin{cor}\label{interautomorf}
The map $g_{(v,\alpha)}$ is an automorphism of $\cal M$, which interchanges simplices
$S_A$, $S_B$.
\end{cor}
We write ${\cal M}^{\nu_t}_{k}$ for the set of all points and blocks of $\cal M$ labelled by the elements of the set $X^{\nu_t}_{k}$, and ${\cal M}^{\nu_t}=\{{\cal M}^{\nu_t}_{k}\colon k\leq m_t\}$.
\begin{lem}\label{autBontoB}
Let $f$ be an automorphism of ${\cal M}$, which
\begin{enumerate}[(1)]\itemsep-2pt
\item\label{caseone}
 maps $B$-blocks onto $B$-blocks, or
\item\label{casetwo}
maps $B$-blocks onto $A$-blocks.
\end{enumerate}
There is $v\in\bigtimes_{t=1}^{r}C_{\nu_t}^{m_t}$ and
$\alpha\in\bigtimes_{t=1}^{r}S_{m_t}$ such that 
\begin{sentences}\itemsep-2pt
\item\label{casef}
$f=f_{(v,\alpha)}$ in case \eqref{caseone}, or
\item\label{caseg}
$f=g_{(v,\alpha)}$ in case \eqref{casetwo}.
\end{sentences}
In particular, for each $k\leq m_t$ there is $k'\leq m_t$ such that $f({\cal M}^{\nu_t}_{k})={\cal M}^{\nu_t}_{k'}$.
\end{lem}
\begin{proof}
(i):
Let $i\in X^{\nu_t}_{k}$. Assume that $f(B_{i})=B_{j}$ for some $j\in I$. According to \eqref{label:points} there is $i_0\in\{0,\ldots,\nu_t-1\}$ such that $B_{i}=(t,k,i_0,-2)$, and $j_0\in\{0,\ldots,\nu_{t'}-1\}$ such that $B_{j}=(t',k',j_0,-2)$ for some $t'\leq r$, $k'\leq m_{t'}$.
Then, by \ref{lem:isoform}\eqref{lem:isoform:ii} we get $f((t,k,i_0,\varepsilon))=(t',k',j_0,\varepsilon)$
for each value of $\varepsilon$.
The unique $B$-block containing $a_{i}=(t,k,i_0,1)$ is $B_{\varphi^{-1}(i)}=(t,k,i_0-1\mod \nu_t,-2)$, and the unique $B$-block containing
$a_{j}$ is $B_{\varphi^{-1}(j)}=(t',k,j_0-1\mod\nu_{t'},-2)$.
Hence, $f$ maps $(t,k,i_0-1\mod \nu_t,-2)$ onto $(t',k',j_0-1\mod\nu_{t'},-2)$, and $f$ maps $(t,k,i_0-1\mod \nu_t,\varepsilon)$ onto $(t',k',j_0-1\mod\nu_{t'},\varepsilon)$ generally. By induction we get 
\begin{ctext}
$f\colon (t,k,i_0-u\mod\nu_t,\varepsilon)\mapsto
(t',k',j_0-u\mod\nu_{t'},\varepsilon)$ for all $u=0,\ldots,\nu_t-1$. 
\end{ctext}
This characterizes the action of $f$ on ${\cal M}^{\nu_t}_{k}$, in particular, 
$f({\cal M}^{\nu_t}_{k})\subseteq{\cal M}^{\nu_{t'}}_{k'}$.
Conversely, $f^{-1}$ maps $B_j$ onto $B_i$. By the reasoning, analogous to this, which has been already done, we come to
$f^{-1}({\cal M}^{\nu_{t'}}_{k'})\subseteq {\cal M}^{\nu_t}_{k}$. Consequently, 
$f({\cal M}^{\nu_t}_{k})={\cal M}^{\nu_{t'}}_{k'}$, and therefore $t'=t$ since $f$ is a bijection.
It provides that $f$ preserves the set ${\cal M}^{\nu_t}$.
We define the map $\alpha\in S_{m_t}$ associated 
with $f\ciach{{\cal M}^{\nu_t}}$ by the formula
\begin{ctext}
 $\alpha\colon k\mapsto k'$ iff $f({\cal M}^{\nu_t}_{k})={\cal M}^{\nu_{t}}_{k'}$, 
\end{ctext}
for all $k,k'\leq m_t$.
%
Set $v_k^t=j_0-i_0\mod\nu_t$. Finally the formula for $f$ is the following:
\begin{ctext}
$f\colon (t,k,i,\varepsilon)\mapsto
(t,\alpha(k),i+v_k^t\mod\nu_{t},\varepsilon)$ for all $i=0,\ldots,\nu_t-1$.
\end{ctext}

(ii):
Based on \ref{autBontoA}, $g_0f$ is an automorphism of $\cal M$, which maps $B$-blocks onto $B$-blocks.
Then, from \ref{autBontoB}\eqref{casef}, $g_0f=f_{(v,\alpha)}$ for some $v\in\bigtimes_{t=1}^{r}C_{\nu_t}^{m_t}$ and
$\alpha\in\bigtimes_{t=1}^{r}S_{m_t}$, and thus $f=g^{-1}_0f_{(v,\alpha)}$. Note that $g^{-1}_0=g_{(\mathbf{1},\id)}$. Consequently, $f=g_{(\mathbf{1},\id)}f_{(v,\alpha)}=
g_0f_{(v+\mathbf{1},\alpha)}=g_{(v+\mathbf{1},\alpha)}$. What is more, $f$ preserves the set ${\cal M}^{\nu_t}$, that follows directly from \eqref{map:inter}.
\end{proof}
Now we characterize automorphisms of ${\goth M}_{(n,\varphi)}$, which can be uniquely decomposed into two
mutually inscribed $n$-simplices.
\begin{thm}\label{autgroup}
Let ${\cal M}={\goth M}_{(n,\varphi)}$ and
$1\leq\nu_1<\ldots<\nu_r$ be the lengths of the cycles in the cycle decomposition of $\varphi\in S_n$.
Assume that either
$n=4$ and $\varphi\neq\id$ contains no cycle of length $2$, or $n\geq 5$. 
Then $\Aut({\cal M})\cong
\bigoplus_{i=1}^{r}\big(C_{2\nu_i}^{m_i}\rtimes S_{m_i}\big)$.
\end{thm}
\begin{proof}
Let $F$ be an automorphism of ${\cal M}$.
By \ref{nonoriginal}, there is no special decomposition of $\cal M$. Thus, $F$ either interchanges $S_A(\cal M)$ with $S_B(\cal M)$ or preserves each of them.
According to \ref{autBontoB} there is $v_0\in\bigtimes_{t=1}^{r}C_{\nu_t}^{m_t}$ and
$\alpha_0\in\bigtimes_{t=1}^{r}S_{m_t}$ such that $F=f_{(v_0,\alpha_0)}$ or $F=g_{(v_0,\alpha_0)}=g_0f_{(v_0,\alpha_0)}$.
Furthermore, every $f_{(v,\alpha)}$, $g_{(v,\alpha)}$ with $v\in\bigtimes_{t=1}^{r}C_{\nu_t}^{m_t}$ and
$\alpha\in\bigtimes_{t=1}^{r}S_{m_t}$ is an automorphism of ${\cal M}$ by \ref{presautomorf}, \ref{interautomorf}.
Since, by \ref{autBontoB}, $F$ preserves each of the sets ${\cal M}^{\nu_t}$, we can restrict the proof to the one fixed set ${\cal M}^{\nu_t}$. Thus, we assume that $i=0,\ldots,\nu_t-1$, $k\leq m_t$. For the simplicity of the notation, we 
will write 
$(\alpha(k),i+v_{\alpha(k)},\varepsilon)$ instead of $(t,\alpha_t(k),i+v_{\alpha_t(k)}^t\mod \nu_{t},\varepsilon)$. Moreover, we identify $f_{(v,\alpha)}$ with $f_{(v,\alpha)}\ciach{{\cal M}^{\nu_t}}$, and $g_{(v,\alpha)}$ with $g_{(v,\alpha)}\ciach{{\cal M}^{\nu_t}}$, so we assume $v\in C_{\nu_t}^{m_t}$, $\alpha\in S_{m_t}$.

Let $w\in C_{\nu_t}^{m_t}$, $\beta\in S_{m_t}$ and note that
\begin{cmath}
f_{(w,\beta)}f_{(v,\alpha)}((k,i,\varepsilon)) = f_{(w,\beta)}((\alpha(k),i+v_k,\varepsilon))= 
(\beta\alpha(k),i+v_k+w_{\alpha(k)},\varepsilon).
\end{cmath}
Let $\phi^{\alpha}\colon S_{m_t}\longrightarrow \Aut(C_{\nu_t}^{m_t})$ be the map defined by 
\begin{ctext}
 $\phi_{\alpha}\colon (v_1,\ldots,v_{m_t})\mapsto(v_{\alpha(1)},\ldots,v_{\alpha(m_t)})$,
\end{ctext}
 Then the formula for the composition of $f_{(v,\alpha)}$ and $f_{(w,\beta)}$ is 
\begin{ctext}
$f_{(w,\beta)}f_{(v,\alpha)}=f_{(v+\phi_{\alpha}(w),\beta\alpha)}$.
\end{ctext}
For $f=f_{(v,\alpha)}$, $g_0$, and $\zeta=1,2$ we get\\
$\begin{array}{l}
(k,i,\zeta)\mappedby{g_0}(k,i-1,-\zeta)\mappedby{{f}}
(\alpha(k),i-1+v_k,-\zeta), \\ 
(k,i,-\zeta)\mappedby{g_0}(k,i,\zeta)\mappedby{{f}}
(\alpha(k),i+v_k,\zeta),\\
(k,i,\zeta)\mappedby{{f}}(\alpha(k),i+v_k,\zeta)\mappedby{g_0}
(\alpha(k),i+v_k-1,-\zeta),\\
(k,i,-\zeta)\mappedby{{f}}(\alpha(k),i+v_k,-\zeta)\mappedby{g_0}
(\alpha(k),i+v_k,\zeta).
\end{array}$ \\
This proves that $g_0$ commutes with $f_{(v,\alpha)}$. 
Note also that
\begin{equation*}
g_0^{z}=
\left\{
\begin{array}{lll}
f_{(-\mathbf{\frac{z}{2}},\id)} & \text{if } z \text{ is even}, \\
g_{(-\mathbf{\frac{z-1}{2}},\id)} & \text{if } z \text{ is odd}.
\end{array}
\right.
\end{equation*}
Let $k'\leq m_{t'}$.
We introduce the family of maps
\begin{equation*}
g_{0_k}((k',i,\varepsilon))=
\left\{
\begin{array}{lll}
g_0((k',i,\varepsilon)) & \text{if } k'=k, \\
(k',i,\varepsilon) & \text{otherwise}.
\end{array}
\right.
\end{equation*}
Then the following equalities hold
\begin{ctext}
$\{g_{0_k}^{z}\colon z=0,\ldots,2\nu_t-1 \text{ and } z \text{ is even}\}=
\{f_{(v,\id)}\colon v_{k'}=0 \text{ for } k'\neq k\}$, \\
$\{g_{0_k}^z\colon z=0,\ldots,2\nu_t-1 \text{ and } z \text{ is odd}\}=
\{g_{(v,\id)}\colon v_{k'}=0 \text{ for } k'\neq k\}$.
\end{ctext}
Therefore, for each $v\in C_{\nu_t}^{m_t}$ we have
 $f_{(v,\id)}=g_{0_1}^{z_{1}}g_{0_2}^{z_{2}}\ldots
g_{0_{m_t}}^{z_{m_t}}$,
where all numbers $z_{k}=0,\ldots,2\nu_t-1$ are even. Likewise
$g_{(v,\id)}=g_{0_1}^{z_{1}}g_{0_2}^{z_{2}}\ldots
g_{0_{m_t}}^{z_{m_t}}$,
 where all numbers $z_{k}$ are odd.
Hence, for each $F\in\Aut({\cal M})$ there is $v\in C_{\nu_t}^{m_t}$ and $\alpha\in S_{m_t}$
such that $F=f_{(v,\id)}f_{(\mathbf{0},\alpha)}$ or $F=g_{(v,\id)}f_{(\mathbf{0},\alpha)}$.
To complete the proof it suffices to determine all suitable compositions: \\
$f_{(v,\id)}f_{(w,\id)}=f_{(w+v,\id)}$, \\
 $f_{(v,\id)}g_{(w,\id)}=f_{(v,\id)}g_0f_{(w,\id)}=g_0f_{(v,\id)}f_{(w,\id)}=
g_0f_{w+v,\id}=g_{w+v,\id}$, \\
$g_{(v,\id)}g_{(w,\id)}=g_0f_{(v,\id)}g_0f_{(w,\id)}=g_0^2f_{(w+v),\id}=
f_{(-\mathbf{1},\id)}f_{(w+v),\id}=f_{(w+v-\mathbf{1}),\id}$,\\
$f_{(\mathbf{0},\alpha)}f_{(\mathbf{0},\beta)}=f_{(\mathbf{0},\beta\alpha)}$, \\
$f_{(v,\id)}f_{(\mathbf{0},\alpha)}=
f_{(\phi^{\alpha}(v),\alpha)}$, and finally \\
$g_{(v,\id)}f_{(\mathbf{0},\alpha)}=g_0f_{(v,\id)}f_{(\mathbf{0},\alpha)}=
g_0f_{(\phi^{\alpha}(v),\alpha)}=g_{(\phi^{\alpha}(v),\alpha)}$.
\end{proof}
The M\"obius $n$-pairs, which automorphism groups are not characterized by \ref{autgroup}, admit a special decomposition.
We say that an automorphism $f$ of a M\"obius $n$-pair ${\cal M}$ {\em yields a special decomposition} of ${\cal M}$ if $f$ maps the pair $\{S_A,S_B\}$ onto a distinct pair of mutually inscribed simplices.
\begin{thm}\label{autgrup4}
The automorphism group of ${\goth M}_{(4,\varphi)}$ is isomorphic to 
\begin{sentences}
\item\label{autgrup4:1}
 $(C_4\oplus S_2)\rtimes C_2$ if $\varphi\in S_4$ contains precisely one cycle of length $2$,
\item\label{autgrup4:2}
 $(C_4^2\rtimes S_2)\rtimes C_2$ if $\varphi\in S_4$ contains two cycles of length $2$.
\end{sentences}
\end{thm}
\begin{proof}
In view of \ref{conclass:strong}, without loss of generality we can consider ${\cal M}_1={\goth M}_{(4,{\varphi}_1)}$ 
with $\varphi_1=(1)(2)(34)$ in case \eqref{autgrup4:1}, and ${\cal M}_2={\goth M}_{(4,\varphi_2)}$ with $\varphi_2=(12)(34)$ in case \eqref{autgrup4:2}
(comp. Figures \ref{levi:moeb_3cyc}, \ref{levi:moeb_2cyc22}). Let $F_s\in\Aut({\cal M}_s)$ for $s=1,2$.
By \ref{cor:all-s-labile}, there is the special decomposition of each of ${\cal M}_s$. Thus, $F_s$  maps the pair $\{S_A,S_B\}$ onto $\{S_A,S_B\}$ or $F_s$ yields the special decomposition of ${\cal M}_s$.
In case $F_s$ maps the pair $\{S_A,S_B\}$ onto $\{S_A,S_B\}$,
by \ref{autBontoB}, there is $v_0\in\{0\}\times\{0\}\times C_2$, 
$\alpha_0\in S_2\times\{\id\}$ for ${\cal M}_1$, or 
$v_0\in C_2\times C_2$, 
$\alpha_0\in S_2$ for ${\cal M}_2$,
such that $F_s=f_{(v_0,\alpha_0)}$ or $F_s=g_{(v_0,\alpha_0)}=g_0f_{(v_0,\alpha_0)}$, respectively for $s=1,2$.
By \ref{presautomorf}, \ref{interautomorf} all maps $F_sf_{(v,\alpha)}$, $F_sg_{(v,\alpha)}$, where  $v\in\{0\}\times\{0\}\times C_2$ and
$\alpha\in S_2\times\{\id\}$ if $s=1$, or $v\in C_2\times C_2$, 
$\alpha\in S_2$ if $s=2$,
 are automorphisms of ${\cal M}_s$ preserving the pair $\{S_A,S_B\}$. Based on the proof of \ref{autgroup}, these maps form the group $C_4\oplus S_2$ if $s=1$, and the group $C_4^2\rtimes S_2$ if $s=2$.
\par
Consider the maps:

\vspace{0.1in}

\noindent
\begin{tabular}{c|cccccccccccccccc}
$x$ & $a_1$ & $a_2$ & $a_3$ & $a_4$ & $b_1$ & $b_2$ & $b_3$ & $b_4$ & $A_1$ & $A_2$ & $A_3$ & $A_4$ & $B_1$ & $B_2$ & $B_3$ & $B_4$ \\
\hline
$\tilde{f}(x)$ & $b_1$ & $b_2$ & $a_4$ & $a_3$ & $a_1$ & $a_2$ & $b_3$ & $b_4$ & $A_2$ & $A_1$ & $B_4$ & $B_3$ & $B_2$ & $B_1$ & $A_4$ & $A_3$ \\
\hline
$\hat{f}(x)$ & $a_2$ & $a_1$ & $b_3$ & $b_4$ & $b_1$ & $b_2$ & $a_3$ & $a_4$ & $B_2$ & $B_1$ & $A_4$ & $A_3$ & $A_1$ & $A_2$ & $B_3$ & $B_4$
\end{tabular}

\vspace{0.1in}

The map $\tilde{f}$ is an automorphism, which yields a special decomposition of ${\cal M}_1$; and $\hat{f}$ is an automorphism, which yields a special decomposition of ${\cal M}_2$. 
Assume that $F_s$ yields a special decomposition of ${\cal M}_s$.
Then $F_1=\tilde{f}F'_1$ and
$F_2=\hat{f}F'_2$, where $F'_s$ is the automorphism of ${\cal M}_s$ given by
\eqref{map:pres} or \eqref{map:inter}.
\par
Let us set the commutativity rules in the automorphism group of ${\cal M}_s$. 
By \eqref{label:points}, the points of ${\cal M}_1$, ${\cal M}_2$ correspond to the sequences $(t,i,k,\varepsilon)$ with 
$\varepsilon=1,-1$. 
Using the convention introduced at the beginning of this paragraph we get 
$t=1,2$, $\nu_1=1$, $\nu_2=2$, $m_1=2$, $m_2=1$ and
$X^1_1=\{1\}$, $X^1_2=\{2\}$, $X^2_1=\{3,4\}$ for ${\cal M}_1$;
$t=1$, $\nu_1=2$, $m_1=2$, and
$X^2_1=\{1,2\}$, $X^2_2=\{3,4\}$ for ${\cal M}_2$.
To avoid any misunderstanding, in case ${\cal M}_2$ we will write $Y^2_1$, $Y^2_2$ instead of $X^2_1$, $X^2_2$ respectively. 
Then $\tilde{f}$ maps the points of ${\cal M}_1$ by the formula: 
\begin{equation*}\label{map:interpairs1}
\tilde{f}((t,k,i,\varepsilon))=
\left\{
\begin{array}{lll}
(t,k,i,-\varepsilon) & \text {for } i+\mu_k^t\in X^1_1, X^1_2,\\
(t,k,i+1\mod 2,\varepsilon) & \text {for } \varepsilon=1,\; i+\mu_k^t\in X^2_1, \\
(t,k,i,\varepsilon) & \text {for } \varepsilon=-1,\; i+\mu_k^t\in X^2_1.
\end{array}
\right.
\end{equation*}
The map $\hat{f}$ can be defined on points of ${\cal M}_2$ as:
\begin{equation*}\label{map:interpairs2}
\hat{f}((t,k,i,\varepsilon))=
\left\{
\begin{array}{lll}
(t,k,i+1\mod 2,\varepsilon) & \text {for } \varepsilon=1,\; i+\mu_k^t\in Y^2_1, \\
(t,k,i,\varepsilon) & \text {for } \varepsilon=-1,\; i+\mu_k^t\in Y^2_1, \\
(t,k,i,-\varepsilon) & \text {for } i+\mu_k^t\in Y^2_2.
\end{array}
\right.
\end{equation*}
In order to simplify the notation we will not add ``$\mod 2$" in the forthcoming calculations, as we believe it will be clear from the context.
Let us establish formulas for the compositions of $\tilde{f}$ and $g_0$ \\
$\begin{array}{ll}
\text{for } i+\mu_k^t\in X^1_1, X^1_2\colon & \\
 (t,k,i,1)\mappedby{g_0}(t,k,i,-1)\mappedby{\tilde{f}} 
(t,k,i,1), & \\ 
 (t,k,i,1)\mappedby{\tilde{f}}(t,k,i,-1)\mappedby{g_0} 
(t,k,i,1), & \\
(t,k,i,-1)\mappedby{g_0}(t,k,i,1)\mappedby{\tilde{f}} 
(t,k,i,-1), & \\
 (t,k,i,-1)\mappedby{\tilde{f}}(t,k,i,1)\mappedby{g_0} 
(t,k,i,-1); & \\
 \text{for } i+\mu_k^t\in X^2_1\colon  & \\
 (t,k,i,1)\mappedby{g_0}(t,k,i-1,-1)\mappedby{\tilde{f}}
(t,k,i-1,-1), & \\ 
(t,k,i,1)\mappedby{\tilde{f}}(t,k,i+1,1)\mappedby{g_0}
(t,k,i,-1), & \\
 (t,k,i,-1)\mappedby{g_0}(t,k,i,1)\mappedby{\tilde{f}}
(t,k,i+1,1), & \\
 (t,k,i,-1)\mappedby{\tilde{f}}(t,k,i,-1)\mappedby{g_0}
(t,k,i,1); &  
\end{array}$ 
\\
and formulas for compositions of $\tilde{f}$ and an automorphism $f=f_{(v,\alpha)}$ of ${\cal M}_1$\\
$\begin{array}{ll}
\text{for } i+\mu_k^t\in X^1_1, X^1_2\colon & \\
(t,k,i,1)\mappedby{f}(t,\alpha(k),i,1)\mappedby{\tilde{f}}
(t,\alpha(k),i,-1),  & \\ 
(t,k,i,1)\mappedby{\tilde{f}}(t,k,i,-1)\mappedby{f}
(t,\alpha(k),i,-1);  & \\
(t,k,i,-1)\mappedby{f}(t,\alpha(k),i,-1)\mappedby{\tilde{f}}
(t,\alpha(k),i,1),  & \\ 
(t,k,i,-1)\mappedby{\tilde{f}}(t,k,i,1)\mappedby{f}
(t,\alpha(k),i,1);  & \\
\text{for } i+\mu_k^t\in X^2_1\colon & \\
(t,k,i,1)\mappedby{f}(t,k,i+v^2_1,1)\mappedby{\tilde{f}}
(t,k,i+v^2_1+1,1),  & \\ 
(t,k,i,1)\mappedby{\tilde{f}}(t,k,i+1,1)\mappedby{f}
(t,k,i+1+v^2_1,1),  & \\
(t,k,i,-1)\mappedby{f}(t,k,i+v^2_1,-1)\mappedby{\tilde{f}}
(t,k,i+v^2_1,-1),  & \\
(t,k,i,-1)\mappedby{\tilde{f}}(t,k,i,-1)\mappedby{f}
(t,k,i+v^2_1,-1).  & 
\end{array}$
\\
This all proves that
\begin{ctext}
$\tilde{f}f=f\tilde{f}$ and $\tilde{f} g_0=\tilde{f} g_{(\mathbf{0},\id,)}=
g_{(\tau_{(0,0,1)}(\mathbf{0}),\id,)}\tilde{f}$,
\end{ctext}
where $\tau_{(0,0,1)}(v)=\tau_{(0,0,1)}((v^1_1,v^1_2,v^2_1))=(v^1_1,v^1_2,v^2_1+1)$. 
Moreover,  $\tilde{f}^2=\id$, and consequently $\{\tilde{f},\id\}=C_2$.
\par
Analogous calculation can be done for $\hat{f}$. Note that the compositions of $\tilde{f}$ and $g_0$ for
$i+\mu_k^t\in X_1^2$, and the compositions of
$\hat{f}$ and $g_0$ for $i+\mu_k^t\in Y_1^2$ coincide. Thus, in ${\cal M}_2$, for $i+\mu_k^t\in Y_1^2$ we need to determine only formulas of compositions of $\hat{f}$ and $f=f_{(v,\alpha)}$.
Namely \\
$\begin{array}{ll}
(t,k,i,1)\mappedby{f}(t,\alpha(k),i+v_1^1,1)\mappedby{\hat{f}} 
\left\{
\begin{array}{ll}
(t,\alpha(k),i+v_1^1+1,1) & \text{if } \alpha=\id, \\
(t,\alpha(k),i+v_1^1,-1) & \text{if } \alpha=(12), 
\end{array}
\right.
& \\ 
(t,k,i,1)\mappedby{\hat{f}}(t,k,i+1,1)\mappedby{f}
(t,\alpha(k),i+1+v_1^1,1),  & \\
(t,k,i,-1)\mappedby{f}(t,\alpha(k),i+v_1^1,-1)\mappedby{\hat{f}}
\left\{
\begin{array}{ll}
(t,\alpha(k),i+v_1^1,-1) & \text{if } \alpha=\id, \\
(t,\alpha(k),i+v_1^1,1) & \text{if } \alpha=(12), 
\end{array}
\right.
 & \\
(t,k,i,-1)\mappedby{\hat{f}}(t,k,i,-1)\mappedby{f}
(t,\alpha(k),i+v_1^1,-1).  & 
\end{array}$ \\
Let $\Delta_{\alpha}$ be the map associated with $\alpha\in S_2$, which acts on the points of ${\cal M}_2$ by the formula
\begin{equation*}
\Delta_{\alpha}((t,k,i,\varepsilon))=
\left\{
\begin{array}{ll}
(t,k,i+1,\varepsilon) & \text{for } \alpha=(12) \text{ and } \varepsilon =1,\\
(t,k,i,\varepsilon) & \text{for } \alpha=(12) \text{ and } \varepsilon =-1, \\
(t,k,i,-\varepsilon) & \text {for } \alpha=\id.
\end{array}
\right.
\end{equation*}
We fix all compositions of $\hat{f}$ with $g_0$, and of $\hat{f}$ with $f$, for $i+\mu_k^t\in Y_2^2$:\\
$\begin{array}{l}
 (t,k,i,1)\mappedby{\hat{f}}(t,k,i,-1)\mappedby{g_0}
(t,k,i,1), \\ 
(t,k,i,1)\mappedby{g_0}(t,k,i-1,-1)\mappedby{\hat{f}}
(t,k,i-1,1),\\
 (t,k,i,-1)\mappedby{\hat{f}}(t,k,i,1)\mappedby{g_0}
(t,k,i-1,-1),\\
(t,k,i,-1)\mappedby{g_0}(t,k,i,1)\mappedby{\hat{f}}
(t,k,i,-1); \\
 (t,k,i,1)\mappedby{f}(t,\alpha(k),i+v^1_2,1)\mappedby{\hat{f}}
\Delta_{\alpha}((t,\alpha(k),i+v^1_2,1)), \\ 
 (t,k,i,1)\mappedby{\hat{f}}(t,k,i,-1)\mappedby{f}
(t,\alpha(k),i+v^1_2,-1), \\
(t,k,i,-1)\mappedby{f}(t,\alpha(k),i+v^1_2,-1)\mappedby{\hat{f}}
\Delta_{\alpha}((t,\alpha(k),i+v^1_2,-1)), \\ 
 (t,k,i,-1)\mappedby{\hat{f}}(t,k,i,1)\mappedby{f}
(t,\alpha(k),i+v^1_2,1),
\end{array}$
\\
Set $\tau_{(1,1)}(v)=\tau_{(1,1)}((v_1^1,v^1_2))=(v_1^1+1,v^1_2+1)$. Finally we obtain \\
$\hat{f} g_0=\hat{f} g_{(\mathbf{0},\id,)}=
g_{(\tau_{(1,1)}(\mathbf{0}),\id)}\hat{f}$, \\ $\hat{f}f_{(v,\alpha)}=f_{(v,\alpha)}\hat{f}$ if only $\alpha=\id$, \\ $\hat{f}f_{(v,\alpha)}=g_0f_{(v,\alpha)}\hat{f}=g_{(v,\alpha)}\hat{f}$ provided that $i+\mu_k^t\in Y_1^2$ and $\alpha=(12)$, \\
$\hat{f}f_{(v,\alpha)}=g_0^{-1}f_{(v,\alpha)}\hat{f}=g_{(\tau_{(1,1)}(v),\alpha)}\hat{f}$ as long as $i+\mu_k^t\in Y_2^2$
and $\alpha=(12)$.\\
Furthermore, $\hat{f}^2=\id$, and thus $\{\hat{f},\id\}=C_2$.

\end{proof}
 
\section*{Acknowledgements}

First and foremost I would like to thank Professor Krzysztof Pra{\.z}mowski for his support and valuable remarks, that considerably improve the paper. I am also grateful to my colleague Mariusz \.Zynel for fruitful discussions and language tips. 

\end{document}